\documentclass[11pt,reqno]{amsart}
\usepackage{hyphenat}

\usepackage[a4paper,margin=27mm]{geometry}
\usepackage{amsmath,amssymb,amsthm,mathtools}
\usepackage{microtype}
\usepackage{xcolor}
\usepackage[colorlinks=true,linkcolor=blue!55!black,citecolor=blue!55!black,urlcolor=blue!55!black,draft]{hyperref}

\numberwithin{equation}{section}
\allowdisplaybreaks

\newtheorem{theorem}{Theorem}[section]
\newtheorem{proposition}[theorem]{Proposition}
\newtheorem{lemma}[theorem]{Lemma}
\newtheorem{corollary}[theorem]{Corollary}
\newtheorem{remark}[theorem]{Remark}

\newcommand{\R}{\mathbb{R}}
\newcommand{\tr}{\operatorname{tr}}
\newcommand{\Ccap}{\mathcal{C}_{\theta}}
\newcommand{\Sph}{\mathbb{S}}
\newcommand{\dd}{\,d\sigma}

\newcommand{\cL}{\mathcal{L}}

\newcommand{\diver}{\operatorname{div}}
\newcommand{\ip}[2]{\left\langle #1,#2\right\rangle}
\newcommand{\norm}[1]{\left|#1\right|}

\title[Capillary Gauss solitons]{Uniqueness of capillary Gauss solitons}

\author[X. Mei]{Xinqun Mei}
\address[X. Mei]{Key Laboratory of Pure and Applied Mathematics, School of Mathematical Sciences, Peking University, Beijing, 100871, P.R. China}
\email{qunmath@pku.edu.cn}
	
\author[G. Wang]{Guofang Wang}
\address[G. Wang]{Mathematisches Institut, Albert-Ludwigs-Universit\"{a}t Freiburg, Freiburg im Breisgau, 79104, Germany}
\email{guofang.wang@math.uni-freiburg.de}

\author[L. Weng]{Liangjun Weng}
\address[L. Weng]{Centro di Ricerca Matematica Ennio De Giorgi, Scuola Normale Superiore, Pisa, 56126, Italy}
\email{liangjun.weng@sns.it}



\begin{document}
\begin{abstract} We prove the rigidity conjecture of \cite[Conjecture~1.2]{MeiWangWeng} for smooth strictly convex capillary Gauss solitons in a Euclidean half-space with an acute contact angle: every such soliton is a spherical cap. Combined with our previous convergence result for the capillary Gauss curvature flow \cite[Theorem~1.1]{MeiWangWeng}, it follows that the flow starting from a strictly convex capillary hypersurface with an acute contact angle converges to a capillary spherical cap, after a suitable rescaling. 
\end{abstract}
\maketitle


\section{Introduction and main results}

In this paper we resolve, in every dimension $n\geq1$ and for acute contact angles $\theta\in(0,\frac\pi2)$, the rigidity problem for capillary Gauss solitons proposed in \cite{MeiWangWeng}. Combined with the convergence theorem for the capillary Gauss curvature flow proved in \cite[Theorem~1.1]{MeiWangWeng}, this shows that the rescaled capillary Gauss curvature flow starting from any smooth strictly convex capillary hypersurface in a half-space with acute contact angle converges to a capillary spherical cap.

The endpoint $\theta=\frac\pi2$ reduces to the closed problem: reflecting an orthogonal capillary hypersurface across the boundary hyperplane turns the capillary soliton equation into the Gauss soliton equation on a closed hypersurface. In the closed setting, rigidity of strictly convex Gauss curvature solitons was proved by Choi--Daskalopoulos \cite{CD}; together with the work of Guan--Ni \cite{GN}, this completes the proof of global convergence of the (closed) Gauss curvature flow for $n\geq 3$. The case $n=2$ was proved earlier by Andrews \cite{andrews1999}. For alternative classification arguments for Gauss soliton, see  \cite{IvakiMilman, Saro2022}; for further background, see \cite{andrews1996, AGN, BCD, Firey, Tso1985} and \cite[Chapters~15--17]{acgl}. The acute-angle case treated here is therefore genuinely a boundary problem and requires a different rigidity mechanism. Our approach does not apply to obtuse angles $\theta>\frac\pi2$; see Remark \ref{rem:acute-angle-necessary}.

The two-step maximum principle argument of Choi--Daskalopoulos \cite{CD} (see also \cite{BCD}) does not seem to adapt directly to capillary Gauss solitons, for two reasons. First, the normalized equation \eqref{eq:normalized-equation} is an anisotropic logarithmic Minkowski-type problem. For the classical logarithmic Minkowski problem, see Böröczky--Lutwak--Yang--Zhang \cite{BLYZ2013}. In the present setting, the anisotropy is induced by a nonconstant weight $\ell$. Uniqueness for the anisotropic logarithmic Minkowski problem remains widely open and may fail without additional assumptions. To the best of our knowledge, in the planar case, Böröczky--Lutwak--Yang--Zhang \cite{BLYZ}  solved the uniqueness for the logarithmic-Minkowski problem for smooth strictly convex origin-symmetric bodies. 
In higher dimensions, it is only known under the assumption that the anisotropy is close to a constant, via approaches different from the two-step maximum principle; see \cite{BS, Chen, HI1}, among others.  Second, the hypersurfaces considered here have a boundary, so any auxiliary quantity used in a maximum principle must satisfy a compatible boundary behavior to control possible boundary extrema. These issues obstruct a straightforward boundary version of the closed argument.

This motivates an integral approach. We develop a capillary analogue of the method of Ivaki--Milman \cite{IvakiMilman}, which replaces pointwise curvature estimates by sharp integral inequalities. The capillary adaptation is not formal: integration by parts produces boundary terms, and the natural test functions dictated by the spherical-cap geometry do not satisfy a common boundary condition.

The main new ingredient is a frame adapted to the capillary boundary. In the spherical-cap formulation, after normalizing the support function by the spherical-cap model, this frame yields $n+1$ distinguished modes, splitting into $n$ Neumann modes and one Dirichlet mode. The Neumann modes can be treated by a boundary analogue of the closed integral argument, using the local capillary Alexandrov--Fenchel inequality established in \cite[Theorem~3.2]{MeiWangWengXia} when $n\geq2$, and the classical planar Minkowski inequality when $n=1$. The Dirichlet mode has no counterpart in the closed problem and is controlled by a separate integral identity. Combining the estimates for the two modes leads to a sharp reverse weighted Poincar\'e inequality, whose equality case forces the normalized support function \eqref{eqn:capillary-support-function} to be constant, equivalently the usual support function $h\equiv \ell$ and hence gives the desired rigidity.

Let $n\geq1$, and let $E_1,\ldots,E_{n+1}$ denote the standard basis of
$\R^{n+1}$. We set
\begin{eqnarray*}
  \R^{n+1}_{+}
  \coloneqq \left\{x\in\R^{n+1}: x_{n+1}=\langle x,E_{n+1}\rangle>0 \right\},
  \qquad e\coloneqq -E_{n+1}.
\end{eqnarray*}
We write $X$ for the position vector, $\nu$ for the outward unit normal,
and $K$ for the Gauss curvature.
A properly embedded smooth compact hypersurface $\Sigma$ in
$\overline{\R^{n+1}_+}$, with boundary
$\partial\Sigma\subset\partial\R^{n+1}_+$, is called a
\emph{capillary hypersurface} if it meets $\partial\R^{n+1}_+$ at a
constant contact angle $\theta\in(0,\pi)$; with our choice of orientation,
this means
\[
  \langle\nu,e\rangle=-\cos\theta
  \qquad\text{on }\partial\Sigma.
\]
When $n=1$, a smooth strictly convex capillary hypersurface is understood
to be a connected embedded arc with everywhere positive curvature which,
together with the segment of $\partial\R^2_+$ joining its endpoints,
bounds a convex body in $\overline{\R^2_+}$. In this case, $K$ is the
curvature of the arc.
A simple example of a capillary hypersurface is the unit spherical cap
\begin{eqnarray*}
  \Ccap \coloneqq \left\{\xi\in\overline{\R^{n+1}_{+}}: |\xi-\cos\theta\,e|=1 \right\}.
\end{eqnarray*}
We denote by $\sigma$, $\nabla$, and $d\sigma$ the round metric,
Levi--Civita connection, and volume element on $\Ccap$, and by $\mu$ the
outward unit conormal along $\partial\Ccap\subset\Ccap$.
The support function of $\Ccap$ is
\[
  \ell(\xi)\coloneqq
  \sin^2\theta+\cos\theta\,\langle\xi,e\rangle \geq 1-|\cos\theta|>0.
\]

We are now in a position to state the main result.
\begin{theorem}
\label{thm:classification}
Let $n\geq1$, and let $\Sigma\subset\overline{\R^{n+1}_{+}}$ be a smooth strictly convex capillary hypersurface with contact angle $\theta\in(0,\frac\pi2)$. Suppose that $\Sigma$ is a capillary Gauss soliton, namely its Gauss curvature satisfies
\begin{equation}
  K
  =\lambda\,
  \frac{\langle X-p,\nu\rangle}{1+\cos\theta\,\langle\nu,e\rangle}
  \qquad \text{on }\Sigma,
  \label{eq:geometric-soliton}
\end{equation}
for some $p\in\partial\R^{n+1}_{+}$ and $\lambda>0$. Then
\begin{equation}
  \Sigma = p+\lambda^{-1/(n+1)}\Ccap.
  \label{eq:classification}
\end{equation}
In other words, the unit spherical cap $\Ccap$ is the unique smooth strictly convex capillary Gauss soliton, up to translations along $\partial\R^{n+1}_{+}$ and scalings.
\end{theorem}

Via the capillary Gauss map parametrization, Theorem~\ref{thm:classification} is reduced to a rigidity problem for the associated support function. After an appropriate normalization, the soliton equation \eqref{eq:geometric-soliton} becomes a Monge--Amp\`ere type equation on $\Ccap$ with a natural Robin boundary condition. The main analytic input is the following uniqueness theorem, which is in principle equivalent to Theorem \ref{thm:classification}.

\begin{theorem}
\label{prop:normalized}
Let $n\geq1$ and $\theta\in\bigl(0,\frac\pi2\bigr)$. Suppose
$h\in C^{\infty}(\Ccap)$ satisfies
$h>0$, $\nabla^2h+h\sigma>0$, and
\begin{eqnarray}
\label{eq:normalized-equation}
\left\{
\begin{array}{rlll}
  h\det(\nabla^2 h+h\sigma) &=& \ell & \text{in }\Ccap,\\
  \nabla_\mu h &=& \cot\theta\,h & \text{on }\partial\Ccap.
\end{array}
\right.
\end{eqnarray}
Then $h=\ell$.
\end{theorem}

Theorem~\ref{prop:normalized} establishes uniqueness for the capillary logarithmic Minkowski problem. Since $\ell$ is nonconstant, this is a genuinely anisotropic rigidity statement.

Combining \cite[Theorem~1.1]{MeiWangWeng} with Theorem~\ref{thm:classification} yields convergence of the capillary Gauss curvature flow, after a suitable rescaling, to a spherical cap. More precisely, we have the following corollary.
\begin{corollary}
Let $n\geq1$, and let $\Sigma_0\subset\overline{\R^{n+1}_{+}}$ be a smooth strictly convex capillary hypersurface with contact angle $\theta\in\bigl(0,\frac\pi2\bigr)$. Set
$\widetilde\nu\coloneqq\nu+\cos\theta\,e$. Then the capillary Gauss
curvature flow is
\begin{eqnarray*}
  \partial_t X(\cdot,t) = -K(\cdot,t)\,\widetilde\nu(\cdot,t).
\end{eqnarray*}
The flow with initial hypersurface $\Sigma_0$ becomes extinct in finite time
$T^*<\infty$. Moreover, after a suitable rescaling, the rescaled
hypersurfaces converge to a spherical cap as the rescaled time tends to
infinity.
\end{corollary}

The main part of the paper is devoted to the proof of Theorem~\ref{prop:normalized}. Once this rigidity result is established, Theorem~\ref{thm:classification} follows by applying it to the support function associated with \eqref{eq:geometric-soliton}. Our argument is inspired by the proof of Ivaki--Milman \cite{IvakiMilman} in the closed setting. The presence of a capillary boundary, however, introduces several new features that require additional arguments. To highlight these differences, we first recall their proof and then outline our argument in the capillary setting.

\bigskip

\noindent{\bf The approach of Ivaki--Milman \cite{IvakiMilman}.}
Let $\Sigma\subset \R^{n+1}$ be a smooth strictly convex closed hypersurface with support function $h$ on $\Sph^{n}$. Its inverse Gauss map can be written as
\[
  X_h\coloneqq h(x)x+\nabla h(x),
  \qquad x\in\Sph^{n}.
\]
Set $A[h]\coloneqq \nabla^{2}h+h\sigma$, let $U^{ij}$ denote the cofactor
tensor of $A[h]$, and write
$dV_h\coloneqq h\det A[h]\,d\sigma$. From the classical 
Alexandrov--Fenchel inequality, Andrews \cite{andrews1997} derived the following local
Brunn--Minkowski (Poincar\'e-type) inequality: 
\begin{equation}\label{eq:closed-spectral-intro}
 \int_{\Sph^{n}} h^{2}U^{ij} f_i f_j\,d\sigma\geq  n\int_{\Sph^{n}} f^{2} dV_h, 
  \qquad \int_{\Sph^{n}} f\,dV_h=0.
\end{equation}
In \cite{IvakiMilman}, Ivaki--Milman apply \eqref{eq:closed-spectral-intro} to the normalized coordinate functions
\begin{eqnarray*}
  \langle X_h,E_A\rangle-
\frac{\int_{\mathbb S^n}\langle X_h,E_A\rangle dV_h}
{\int_{\mathbb S^n}dV_h}, \qquad 1\leq A\leq n+1.
\end{eqnarray*}
Summing over $A$ and invoking the following identities
\begin{equation}\label{eqn:ivaki-milman-identities}
  \sum_{A=1}^{n+1}\langle X_h,E_A\rangle^{2}
  = h^{2}+|\nabla h|^{2},
  \qquad
  \sum_{A=1}^{n+1} d\langle X_h,E_A\rangle\otimes d\langle X_h,E_A\rangle
  = A[h]^{2},
\end{equation}
they obtain a sharp reverse Poincar\'e  inequality for the solution of the Gauss soliton. Combined with the sharp Poincar\'e inequality on $\Sph^{n}$, this yields the uniqueness result in the closed setting.

\bigskip

\noindent{\bf Modification for the capillary setting.}
In the capillary setting, the boundary breaks the ambient symmetry. Set
$u\coloneqq h/\ell$ and
$\xi_\alpha\coloneqq\langle\xi,E_\alpha\rangle$ for
$1\leq\alpha\leq n$. We replace the standard coordinate functions by the
boundary-adapted quantities
\begin{eqnarray}\label{eq:F-horizontal}
  F_{\alpha}
  \coloneqq \left\langle X,\,E_{\alpha}-\cos\theta\,\ell^{-1}\xi_{\alpha}e\right\rangle,
~~1\le \alpha\le n,~~~
  F_{n+1}
  \coloneqq -\sin\theta\,\ell^{-1}\langle X,e\rangle.
\end{eqnarray}

For the functions $F_\alpha$ $(1\le \alpha\le n)$, we apply the Poincar\'e inequality  with the Neumann boundary condition  
\begin{equation}
  \int_{\Ccap}h^2U^{ij}f_i f_j\dd
  \ge n
  \int_{\Ccap}f^2\ell  \dd, \quad \int_{\Ccap}f\ell d\sigma=0, \qquad \nabla_\mu f=0 \text{ on } \partial
  \Ccap,
  \label{eq:local-af2}
\end{equation}
to the modified functions $$ F_\alpha- \frac{\int_{\Ccap} F_\alpha \ell \dd}{\int_{\Ccap} \ell \dd},\qquad 1\leq \alpha\leq n,$$
which is established in Lemma~\ref{lem:local-af} (see Theorem~\ref{thm-mwwx}; for $n\geq2$, see also \cite{MeiWangWengXia}).
The remaining function $F_{n+1}$ is treated instead by the  Poincar\'e inequality with the Dirichlet boundary condition, \eqref{eq:Dirichlet},  proved in Lemma~\ref{lem:hardy}
\begin{equation}\label{eq:Dirichlet_0}
  \int_{\Ccap} h^{2}U^{ij}f_i f_j\,\dd
  \ge n\int_{\Ccap} f^{2}\ell \,\dd
+ \cos\theta\int_{\Ccap}\frac{\tr(U)}{\rho} (hf)^{2}\,\dd, \quad f=0 \hbox{ on }\partial\Ccap.
\end{equation}
Together with the identities
\begin{equation*}
\sum_{A=1}^{n+1}F_A^{2}
=\ell^{2}\norm{\nabla u}^{2}+\sin^2\theta u^{2}, \qquad \sum_{A=1}^{n+1}d(F_A) \otimes d(F_A)=B^2,
\end{equation*}
where $B$ is defined by
\[
  B\coloneqq A[h]- \ell^{-1} {\cos\theta}\,
  \langle X,e\rangle\,\sigma,
\]
we obtain the reverse weighted Poincar\'e inequality \eqref{eq:reverse-poincare}. Combining this with the sharp Neumann Poincar\'e inequality \eqref{eq:model-poincare} yields the desired uniqueness. We remark that the last term in \eqref{eq:Dirichlet_0} requires special care and will be controlled using the positivity of $\cos\theta$. Thus, the assumption $\theta<\frac{\pi}{2}$ is essential to the proof. 

This Neumann--Dirichlet split is the key new feature relative to the closed case. In the boundaryless setting, the $n+1$ coordinate functions enter symmetrically and can be treated by using a single sharp inequality.
In contrast, in the capillary setting the fixed supporting hyperplane singles out one vertical direction, leaving $n$ horizontal translation directions. Accordingly, the (modified) horizontal coordinate functions satisfy Neumann boundary conditions, whereas the vertical coordinate function satisfies a Dirichlet boundary condition.
The corresponding Neumann Poincar\'e inequality \eqref{eq:local-af} follows
from \cite{MeiWangWengXia} when $n\geq2$, and from the classical planar
Minkowski inequality when $n=1$, while its Dirichlet counterpart,
\eqref{eq:Dirichlet}, is proved in Lemma~\ref{lem:hardy}. Both inequalities,
especially the Dirichlet one, precisely capture the effects of the capillary boundary
and lead to the rigidity result.

\bigskip

\noindent\textit{Organization of the rest of the paper.} In Section~\ref{sec2}, we introduce the capillary Gauss map parametrization and develop the associated capillary frame. In Section~\ref {sec:3-spectral}, we establish 
sharp Poincar\'e inequalities for the Neumann boundary and for the Dirichlet boundary. 
Finally, in Section~\ref{sec4}, we combine these inequalities to complete the proof of Theorem~\ref{prop:normalized}, from which Theorem~\ref{thm:classification} follows immediately.

\section{Capillary hypersurfaces}\label{sec2}

Let $\Sigma\subset\overline{\R^{n+1}_{+}}$ ($n\geq1$) be a smooth strictly convex capillary hypersurface with contact angle $\theta\in(0,\pi)$. The capillary Gauss map
\[
  \widetilde\nu\coloneqq \nu+\cos\theta\,e:
  \Sigma\to\Ccap
\]
is a diffeomorphism. For $n\geq2$, this follows from \cite[Lemma~2.2]{MeiWangWengXia}. When $n=1$, it follows directly from the normal-angle parametrization: positive curvature makes the outward normal angle strictly monotone along the arc, and the contact-angle condition gives the endpoint values $-\theta$ and $\theta$. By \cite[Lemma~2.4]{MeiWangWengXia},
using the inverse capillary Gauss map parametrization, the support function is
\begin{equation*}
h(\xi)
  = \langle X(\xi),\nu(X(\xi))\rangle
  = \langle X(\xi),\xi-\cos\theta\,e\rangle.
\end{equation*}
It satisfies the Robin boundary condition
\begin{equation}\label{eq:robin}
  \nabla_{\mu}h=\cot\theta\,h
  \qquad\text{on }\partial\Ccap.
\end{equation}
For $f\in C^{2}(\Ccap)$, we set
\begin{equation*}
  A[f]\coloneqq \nabla^{2}f+f\sigma.
\end{equation*}
The inverse capillary Gauss map parametrization and Gauss curvature are
\begin{equation}\label{eq:inverse-gauss}
  X=\nabla h(\xi)+h(\xi)(\xi-\cos\theta\,e),
  \qquad
  K=(\det A[h])^{-1}.
\end{equation}
Here and below, determinants and traces are taken with respect to $\sigma$.
We adopt the Einstein summation convention for repeated indices and raise
and lower indices with $\sigma$. Thus, for a symmetric $(0,2)$-tensor $B$,
we write
\[
  B_j{}^k\coloneqq \sigma^{k\ell}B_{j\ell},
  \qquad
  (B^2)_{ij}\coloneqq B_i{}^kB_{kj}.
\]

\subsection{The standard model}
Recall that the support function of the unit spherical cap $\Ccap$ is
\begin{equation*} 
  \ell(\xi)\coloneqq \sin^{2}\theta+\cos\theta\,\langle\xi,e\rangle.
\end{equation*}
It satisfies the same Robin boundary condition as $h$ and
\begin{equation*} 
  A[\ell]=\sigma.
\end{equation*}
We introduce the horizontal coordinate functions and the boundary defining function
\begin{equation*}
  \xi_{\alpha}\coloneqq \langle\xi,E_{\alpha}\rangle,
  \quad 1\le \alpha\le n,
  \qquad
  \rho\coloneqq -\langle\xi,e\rangle=\xi_{n+1}.
\end{equation*}
Then
\begin{equation}\label{eq:ell-rho}
  \ell=\sin^{2}\theta-\cos\theta\,\rho,
  \qquad
  \sin^{2}\theta=\ell+\cos\theta\,\rho,
\end{equation}
and the defining equation of $\Ccap$ gives
\begin{equation}\label{eq:sphere-coordinate-identity}
  \sum_{\alpha=1}^{n}\xi_{\alpha}^{2}+(\rho+\cos\theta)^{2}=1.
\end{equation}
Spherical differentiation yields
\begin{equation}\label{eq:model-hessians}
  A[\xi_{\alpha}]=0,
  \qquad
  A[\rho]=-\cos\theta\,\sigma,
\end{equation}
and on $\partial\Ccap$,
\begin{equation}\label{qeq:robin-bry}
\nabla_{\mu}\xi_{\alpha}=\cot\theta\,\xi_{\alpha},
  \qquad
  \rho=0,
  \qquad
  \nabla_{\mu}\rho=-\sin\theta.
\end{equation}

\medskip

\noindent\textit{The one-dimensional case.}
When $n=1$, we use the normal-angle parametrization
\[
  \Ccap
  = \left\{\xi(s)\coloneqq (\sin s,\cos s-\cos\theta):-\theta\leq s\leq\theta\right\}.
\]
Then the metric, area element, horizontal and vertical coordinates reduce to
\[
  \sigma=ds^2,\qquad d\sigma=ds,\qquad
  \xi_1=\sin s,\qquad
  \rho=\cos s-\cos\theta.
\]
Moreover,
\[
  \ell=1-\cos\theta\cos s,
  \qquad A[h]=h''+h.
\]
Since the outward conormal is $\mu=-\partial_s$ at $s=-\theta$ and
$\mu=\partial_s$ at $s=\theta$, the Robin boundary condition \eqref{eq:robin} is
\begin{eqnarray}
      h'(-\theta)=-\cot\theta\,h(-\theta),
  \qquad
  h'(\theta)=\cot\theta\,h(\theta).
\end{eqnarray}
Thus the tensorial formulas below reduce to their scalar counterparts when
$n=1$.

\medskip
For later use in every dimension, write
$\Delta\coloneqq\tr_{\sigma}\nabla^2$ for the spherical Laplacian. Since
$\nabla^{2}\xi_{\alpha}=-\xi_{\alpha}\,\sigma$ and
$\nabla^{2}\rho+\rho\,\sigma=-\cos\theta\,\sigma$, taking traces gives
\begin{equation}\label{eq:model-laplacians}
  \Delta\xi_{\alpha}=-n\,\xi_{\alpha},
  \qquad
  \Delta\rho=-n(\rho+\cos\theta).
\end{equation}
Thus each $\xi_\alpha$ is a Robin eigenfunction with eigenvalue $n$,
whereas $\rho$ is not a Dirichlet eigenfunction unless
$\theta=\pi/2$.

\subsection{The capillary soliton}\label{subsec:2.2}
Define the cofactor tensor of $A[h]$ by
\begin{equation}\label{eq:U-def}
  U^{ij}
  \coloneqq \frac{\partial\det A[h]}{\partial A[h]_{ij}}
  =\det A[h]\,(A[h]^{-1})^{ij}.
\end{equation}
When $n=1$, $A[h]$ is scalar and \eqref{eq:U-def} reduces to $U^{11}=1$, so \eqref{eq:U-div-free} below is automatic.
Since $A[h]$ is Codazzi, $U$ is divergence-free, that is,
\begin{equation}\label{eq:U-div-free}
  \sum\limits_{i=1}^{n}\nabla_iU^{ij}=0.
\end{equation}
Tangential differentiation of \eqref{eq:robin} yields the boundary block condition
\begin{equation}\label{eq:block-boundary}
  A[h](\mu,\eta)=0
  \qquad\text{on }\partial\Ccap,
  \quad \eta\in T\partial\Ccap.
\end{equation}
Consequently, $A[h]^{-1}$ and $U$ have the same tangential/normal block decomposition along $\partial\Ccap$.
When $n=1$, $T\partial\Ccap=\{0\}$, so \eqref{eq:block-boundary} is vacuous; the asserted block decomposition is automatic.

From now on, assume that \(h\) is a solution of Eq. \eqref{eq:normalized-equation}. 
We now normalize it by the model solution and set
\begin{equation}\label{eqn:capillary-support-function}
  u\coloneqq \frac{h}{\ell}.
\end{equation}
Eq. \eqref{eq:normalized-equation} implies that $u$ satisfies
\begin{equation*} 
\left\{
\begin{array}{rlll}\vspace{2mm}
  \det(A[h]) &=& \frac{1}{u} & \text{in }\Ccap,\\
  \nabla_{\mu}u &=& 0 & \text{on }\partial\Ccap,
\end{array}
\right.
\end{equation*}
where
\begin{equation}\label{eq:A-h-u}
  A[h]=\ell\nabla^{2}u+d\ell\otimes du+du\otimes d\ell+u\sigma.
\end{equation}
The quotient $u$ offers two advantages: the model solution becomes $u\equiv 1$, and the boundary condition becomes a homogeneous Neumann condition.

\begin{lemma}
For $h=\ell u$ with $\nabla_{\mu}u=0$ on $\partial\Ccap$, we have
\begin{equation}\label{eq:trace-identity}
  \int_{\Ccap} h\,\tr A[h]\,d\sigma
  = n\int_{\Ccap} u^{2}\ell\,d\sigma
  - \int_{\Ccap} \ell^{2}|\nabla u|^{2}\,d\sigma.
\end{equation}
\end{lemma}

\begin{proof}
Taking trace in \eqref{eq:A-h-u} w.r.t. $\sigma$ gives
\[
  \tr A[h]=\ell\Delta u+2\langle\nabla\ell,\nabla u\rangle+nu.
\]
Multiplying by $h=\ell u$, integrating by parts, and using the fact $\nabla_{\mu}u=0$, we obtain \eqref{eq:trace-identity}.
\end{proof}

\subsection{The capillary frame}
The capillary boundary conditions determine the basic scalar modes. Since $\xi_{\alpha}$ and $\ell$ satisfy the same Robin boundary condition, while $\rho$ vanishes on the boundary, we have
\begin{equation}\label{eq:basic-boundary-quotients}
  \nabla_{\mu}\left(\ell^{-1}{\xi_{\alpha}}\right)=0,
  \quad 1\le \alpha\le n,
  \qquad
 \ell^{-1}  {\rho}=0
  \quad\text{on }\partial\Ccap.
\end{equation}
The remaining requirement is the coordinate completeness used in the closed proof. Combining \eqref{eq:ell-rho} and \eqref{eq:sphere-coordinate-identity}, we obtain 
\begin{equation}\label{eq:ell-coordinate-identity}
  \ell^{2}=\sin^{2}\theta\sum_{\alpha=1}^{n}\xi_{\alpha}^{2}+\rho^{2}.
\end{equation}
This identity determines the relative normalization of the horizontal and vertical modes.

Define
\begin{equation}\label{eq:omega-def}
 \omega_{\alpha}\coloneqq\frac{\sin\theta\xi_{\alpha}}{\ell},
 \quad 1\leq\alpha\leq n,
 \qquad
 \omega_{n+1}\coloneqq\frac{\rho}{\ell},
\end{equation}
and write $\omega=(\omega_1,\ldots,\omega_{n+1})$.

\begin{lemma}
The functions in \eqref{eq:omega-def} satisfy
\begin{equation}\label{eq:omega-completeness}
 \sum_{A=1}^{n+1}\omega_A^{2}=1,
 \qquad
 \sum_{A=1}^{n+1}\nabla_i\omega_A\nabla_j\omega_A
 =\frac{\sin^{2}\theta}{\ell^{2}}\sigma_{ij},
\end{equation}
and
\begin{equation}\label{eq:omega-orthogonality}
 \sum_{A=1}^{n+1}\omega_A\nabla\omega_A=0.
\end{equation}
On the boundary $\partial\Ccap$,
\begin{equation}\label{eq:omega-boundary}
 \nabla_{\mu}\omega_{\alpha}=0,
 \quad 1\leq\alpha\leq n,
 \qquad
 \omega_{n+1}=0.
\end{equation}
\end{lemma}

\begin{proof}
The first identity in \eqref{eq:omega-completeness} is exactly \eqref{eq:ell-coordinate-identity} divided by $\ell^{2}$, and \eqref{eq:omega-orthogonality} is its differential. For the second identity, differentiate \eqref{eq:omega-def} and use
\[
 d\ell=-\cos\theta\,d\rho,
 \qquad
 \sum_{\alpha=1}^{n}\xi_{\alpha}\,d\xi_{\alpha}=-(\rho+\cos\theta)\,d\rho,
 \qquad
 \sum_{\alpha=1}^{n}d\xi_{\alpha}\otimes d\xi_{\alpha}
 +d\rho\otimes d\rho=\sigma.
\]
A direct substitution gives the stated formula. Finally, \eqref{eq:omega-boundary} follows from \eqref{eq:basic-boundary-quotients}.
\end{proof}

\begin{remark}
The first identity in \eqref{eq:omega-completeness} says that the map $\omega:\Ccap\to \R^{n+1}$ is sphere-valued, namely $\omega(\Ccap)\subset \Sph^n$. This is only a convenient way to package the scalar functions in \eqref{eq:omega-def}. The argument relies solely on the two Parseval-type identities in \eqref{eq:omega-completeness}, the orthogonality relation \eqref{eq:omega-orthogonality}, and the boundary decomposition \eqref{eq:omega-boundary}; no additional geometric structure is introduced. 
\end{remark}

In the closed problem, the ambient coordinate functions of the inverse Gauss map are already admissible test functions. In the capillary setting, however, the horizontal coordinate functions must be modified to satisfy the boundary condition, whereas the vertical coordinate naturally serves as a height function.

We set
\begin{equation}\label{eq:zeta-B-def}
 \zeta\coloneqq  {\cos\theta}\ell^{-1}\ip{X}{e},
 \qquad
 B\coloneqq A[h]-\zeta\sigma.
\end{equation}
From \eqref{eq:F-horizontal}, we have  $F_{n+1}=-\tan\theta\zeta$. The tensor $B$ is the curvature-radius tensor corrected by the same vertical height which appears in the last coordinate.

\begin{proposition}
\label{prop:frame}
Let $F_A$ be given by \eqref{eq:F-horizontal}. Then the following statements hold.

\smallskip
\noindent\emph{(i) Coordinate representation and zeroth-order completeness.}
\begin{equation}\label{eq:T-def}
 F_A=\sin\theta\,u\omega_A
 +\frac{\ell^{2}}{\sin\theta}\ip{\nabla u}{\nabla\omega_A},
 \qquad 1\leq A\leq n+1.
\end{equation}
Consequently, for every constant $c\in\mathbb{R}$,
\begin{equation}\label{eq:F-equals-T}
 F_A-c\sin\theta\,\omega_A
 =\sin\theta\,(u-c)\omega_A
 +\frac{\ell^{2}}{\sin\theta}\ip{\nabla u}{\nabla\omega_A},
 \qquad 1\leq A\leq n+1,
\end{equation}
and
\begin{equation}\label{eq:T-zero-completeness}
 \sum_{A=1}^{n+1}\bigl(F_A-c\sin\theta\,\omega_A\bigr)^{2}
 =\ell^{2}\norm{\nabla u}^{2}+\sin^{2}\theta (u-c)^{2}.
\end{equation}
In particular, taking $c=0$,
\begin{equation}\label{eq:zero-completeness}
 \sum_{A=1}^{n+1}F_A^{2}
 =\ell^{2}\norm{\nabla u}^{2}+\sin^2\theta u^{2}.
\end{equation}

\smallskip
\noindent\emph{(ii) First-order completeness.}
\begin{equation}\label{eq:F-gradient}
 \nabla_iF_A=\frac{\ell}{\sin\theta}B_{ij}\nabla^{j}\omega_A,
 \qquad 1\leq A\leq n+1,
\end{equation}
and consequently
\begin{equation}\label{eq:first-completeness}
 \sum_{A=1}^{n+1}(F_A)_i(F_A)_j=B_{ik}B_j{}^{k}.
\end{equation}

\smallskip
\noindent\emph{(iii)  Boundary condition.}  

\begin{equation}\label{eq:F-boundary}
 \nabla_{\mu}F_{\alpha}=0,
 \quad 1\leq\alpha\leq n,
 \qquad
 F_{n+1}=0
 \quad\text{on }\partial\Ccap.
\end{equation}
\end{proposition}

\begin{proof}

We first derive the coordinate representation. From \eqref{eq:inverse-gauss}, together with $h=\ell u$ and $d\ell=-\cos\theta\,d\rho$, we obtain
\begin{equation*} 
  \langle X,e\rangle=-\rho u-\ell\,\langle\nabla\rho,\nabla u\rangle.
\end{equation*}
Using \eqref{eq:sphere-coordinate-identity} to evaluate the scalar products involving $\nabla\xi_{\alpha}$ and $\nabla\rho$, a direct substitution in \eqref{eq:F-horizontal} gives
\[
 F_{\alpha}
 =\sin\theta\,u\omega_{\alpha}+\frac{\ell^{2}}{\sin\theta}
 \ip{\nabla u}{\nabla\omega_{\alpha}},
 \qquad
 F_{n+1}
 =\sin\theta\,u\omega_{n+1}+\frac{\ell^{2}}{\sin\theta}
 \ip{\nabla u}{\nabla\omega_{n+1}}.
\]
This proves \eqref{eq:T-def}. Subtracting $c\sin\theta\,\omega_A$ from both sides gives \eqref{eq:F-equals-T}.

We now square \eqref{eq:F-equals-T} and sum over $A$. The mixed term vanishes by \eqref{eq:omega-orthogonality}, while the two diagonal terms are evaluated by \eqref{eq:omega-completeness}. Hence
\[
 \sum_{A=1}^{n+1}\bigl(F_A-c\sin\theta\,\omega_A\bigr)^2
 =\sin^2\theta (u-c)^{2}+\frac{\ell^{4}}{\sin^2\theta}
 \sum_{A=1}^{n+1} \ip{\nabla u}{\nabla\omega_A}^{2}
 =\sin^2\theta (u-c)^{2}+\ell^{2}\norm{\nabla u}^{2}.
\]
This proves \eqref{eq:T-zero-completeness}; taking $c=0$ gives \eqref{eq:zero-completeness}.

We next compute the first derivatives of the modified coordinate functions.
Writing $D$ for the Euclidean connection, the standard inverse Gauss map
identity gives
\begin{equation*} 
  D_iX=A[h]_i{}^{j}\nabla_j\xi.
\end{equation*}
Therefore
\[
  \nabla_i\langle X,E_{\alpha}\rangle
  =A[h]_i{}^{j}\nabla_j\xi_{\alpha},
  \qquad
  \nabla_i\langle X,e\rangle
  =\frac{1}{\cos\theta}A[h]_i{}^{j}\nabla_j\ell.
\]
Using \eqref{eq:zeta-B-def}, we find
\begin{equation}\label{eq:F-alpha-gradient-direct}
  \nabla_iF_{\alpha}
  =\ell\,B_i{}^{j}\nabla_j\left( \frac{\xi_{\alpha}}{\ell}\right)
  =\frac{\ell}{\sin\theta}B_i{}^{j}\nabla_j\omega_{\alpha}.
\end{equation}
Similarly,
\begin{equation}\label{eq:zeta-gradient}
  \ell\nabla_i\zeta=B_i{}^{j}\nabla_j\ell.
\end{equation}
Since
\begin{equation}\label{eq:rho-ell-gradient}
  \nabla\left(\frac{\rho}{\ell}\right)
  =-\frac{\sin^{2}\theta}{\cos\theta\,\ell^{2}}\nabla\ell,
\end{equation}
the identities \eqref{eq:zeta-gradient} and \eqref{eq:rho-ell-gradient}, together with $F_{n+1}=-\tan\theta \zeta$, yield the case $A=n+1$ of \eqref{eq:F-gradient}. Summing the products of \eqref{eq:F-gradient} over $A$ and applying the second identity in \eqref{eq:omega-completeness}, we prove \eqref{eq:first-completeness}.

It remains to verify the boundary conditions. On
$\partial\Ccap$, the function $\xi_{\alpha}/\ell$ satisfies the homogeneous
Neumann boundary condition by \eqref{eq:basic-boundary-quotients}; hence
$\nabla(\xi_{\alpha}/\ell)$ is tangential. When $n=1$, the tangential space
of $\partial\Ccap$ is zero, so $\nabla(\xi_1/\ell)=0$ at both endpoints and
\eqref{eq:F-alpha-gradient-direct} directly gives
$\nabla_\mu F_1=0$. When $n\geq2$, the same conclusion follows from the
block condition \eqref{eq:block-boundary} and
\eqref{eq:F-alpha-gradient-direct}.

On $\partial\Ccap$, one has
\[
  e=\sin\theta\,\mu-\cos\theta(\xi-\cos\theta\,e).
\]
Using \eqref{eq:robin} and \eqref{eq:inverse-gauss}, 
\[
  \langle X,e\rangle=\sin\theta\,\nabla_{\mu}h-\cos\theta\,h=0.
\]
Thus, $F_{n+1}=0$ on $\partial\Ccap$.
\end{proof}

\begin{remark}
The ordinary coordinate function $\langle X,E_{\alpha}\rangle$ does not, in general, satisfy a Neumann boundary condition. The definition \eqref{eq:F-horizontal} adds a vertical multiple of $\langle X,e\rangle$, with coefficient fixed by the function $\xi_{\alpha}/\ell$. The resulting derivative \eqref{eq:F-alpha-gradient-direct} becomes exactly $\ell\,B\,\nabla(\xi_{\alpha}/\ell)$, so both the Neumann boundary condition and the tensor completeness are preserved. In particular, the correction is forced by the boundary geometry; it is not chosen to cancel a later term.
\end{remark}

\section{Sharp  Poincar\'e  inequalities}\label{sec:3-spectral}

\subsection{The elliptic operator and boundary spectrum}
Throughout this section, let $\theta\in(0,\frac\pi2)$ and let
$h\in C^\infty(\Ccap)$ satisfy
\[
  h>0,\qquad A[h]>0,\qquad
  \nabla_\mu h=\cot\theta\,h
  \quad\text{on }\partial\Ccap.
\]
For $f\in C^{2}(\Ccap)$, define
\begin{equation*}
  \cL_h f
  \coloneqq -\frac{1}{h\det A[h]}\,\nabla_j\bigl(h^{2}U^{ij}\nabla_i f\bigr).
\end{equation*}
 
The following identity relates $\cL_h$ to the multiplicative linearization of the Monge--Amp\`ere operator.

\begin{lemma} 
\label{lem:spectral-dictionary}
For every $f\in C^{2}(\Ccap)$, we have
\begin{equation}\label{eq:spectral-dictionary}
  \tr\bigl(A[h]^{-1}A[hf]\bigr)=nf-\cL_hf.
\end{equation}
\end{lemma}

\begin{proof}
The product rule gives
\begin{equation}\label{eq:A-product}
 A[hf]=fA[h]+h\nabla^{2}f+dh\otimes df+df\otimes dh.
\end{equation}
Using \eqref{eq:U-div-free} and \eqref{eq:U-def},
\[
 \cL_hf
 =-h(A[h]^{-1})^{ij}f_{ij}
 -2(A[h]^{-1})^{ij}h_i f_j.
\]
Taking the $A[h]^{-1}$-trace of \eqref{eq:A-product} proves
\eqref{eq:spectral-dictionary}.
\end{proof}

We now apply this identity to the quotient functions underlying the adapted frame. Set
\begin{equation}\label{eq:basic-modes}
 \psi_{\alpha}\coloneqq \frac{\xi_{\alpha}}{h},
 \quad 1\leq\alpha\leq n,
 \qquad
 \psi_{n+1}\coloneqq \frac{\rho}{h}.
\end{equation}
Since $h\psi_{\alpha}=\xi_{\alpha}$ and $h\psi_{n+1}=\rho$,
Lemma~\ref{lem:spectral-dictionary}, \eqref{eq:model-hessians}, and
\eqref{qeq:robin-bry} imply
\begin{equation*}
\left\{
\begin{array}{rlll}
  \cL_h\psi_{\alpha} &=& n\psi_{\alpha} & \text{in }\Ccap,\\
  \nabla_{\mu}\psi_{\alpha} &=& 0 & \text{on }\partial\Ccap,
\end{array}
\right.
\end{equation*}
and
\begin{equation*}
\left\{
\begin{array}{rlll}
  (\cL_h-n)\psi_{n+1} &=& \cos\theta\,\tr(A[h]^{-1}) & \text{in }\Ccap,\\
  \psi_{n+1} &=& 0 & \text{on }\partial\Ccap.
\end{array}
\right.
\end{equation*} 
Hence
$\psi_\alpha$ ($\alpha =1,\ldots, n$) are the Neumann eigenfunctions of $\mathcal{L}_h$, while $\psi_{n+1}$ is the Dirichlet eigenfunction of $\mathcal{L}_h$ if and only if $\theta=\pi/2$.

\subsection{A  Poincar\'e  inequality for the Neumann boundary condition}
Following \cite[Definition~2.5]{MeiWangWengXia}, a function
$f\in C^2(\Ccap)$ is called a \emph{capillary function} if it satisfies the
Robin boundary condition
\begin{equation*}
  \nabla_{\mu}f=\cot\theta\,f
  \qquad\text{on }\partial\Ccap.
\end{equation*}
A capillary function $f$ is called \emph{capillary convex} if $A[f]$ is
positive definite.  
Let $Q$ denote the mixed discriminant,
normalized by $Q(A,A,\ldots,A)=\det A$. For capillary functions
$f_0,f_1,\ldots,f_n$, we recall the capillary mixed volume from
\cite[Eq.~(2.9)]{MeiWangWengXia}:
\begin{equation}\label{eq:mixed-volume}
  V(f_0,f_1,\ldots,f_n)
  \coloneqq \frac{1}{n+1}\int_{\Ccap}
  f_0\,Q\bigl(A[f_1],A[f_2],\ldots,A[f_n]\bigr)\,d\sigma.
\end{equation}
The symmetry of \eqref{eq:mixed-volume} follows from integration by parts
and the Robin boundary condition. When $n=1$, one has $Q(A)=A$, so
\eqref{eq:mixed-volume} becomes
\begin{eqnarray}\label{eqn:n=1-mixed-n-volume}
      V(f,f_1)=\frac12\int_{-\theta}^{\theta}f(f_1''+f_1)\,ds.
\end{eqnarray}

For $n\geq2$, we use the capillary Alexandrov--Fenchel inequality from \cite[Theorem~3.2]{MeiWangWengXia}. The planar case $n=1$, which follows from the classical Minkowski inequality \cite[Theorem~7.2.1, Eq.~(7.18)]{Schneider2014}, is restated below in the capillary setting.

\begin{theorem}\label{thm-mwwx}
Let $n\geq1$, and let $f_{1},\ldots,f_{n}\in C^{2}(\Ccap)$ be capillary convex functions. If $n=1$, assume in addition that $f_1>0$. If $n\geq2$, assume that at least one of $f_{k}$ ($2\le k\le n$) is positive. Then for any capillary function $f\in C^{2}(\Ccap)$,
\begin{equation}\label{eq:function-AF}
  V^{2}(f,f_{1},f_{2},\ldots,f_{n})
  \ge V(f,f,f_{2},\ldots,f_{n})\,V(f_{1},f_{1},f_{2},\ldots,f_{n}).
\end{equation}
Here the arguments $f_2,\ldots,f_n$ are omitted when $n=1$.
Equality holds if and only if
\begin{eqnarray}\label{eqn:af-ineq-equality}
      f=af_{1}+\sum_{\alpha=1}^{n}a_{\alpha}\xi_\alpha,
\end{eqnarray}
for some constants $a,a_{\alpha}\in\R$, $1\leq \alpha\leq n$.
\end{theorem}

\begin{proof}
For $n\geq 2$, the result follows from \cite[Theorem~3.2]{MeiWangWengXia}. The remaining case $n=1$ is proved in the Appendix \ref{appendix}.
\end{proof}

Following Andrews' treatment of the closed case \cite{andrews1997}, we have 
the following local form of the capillary Alexandrov--Fenchel inequality, which is the second-variation form of \eqref{eq:function-AF}.
\begin{lemma}
\label{lem:local-af} 
If $f\in C^2(\Ccap)$ and $\nabla_\mu f=0$ on $\partial\Ccap$, then
\begin{equation}
  \int_{\Ccap}h^2U^{ij}f_i f_j\dd
  \ge n\left\{
  \int_{\Ccap}f^2h \det(A[h]) \dd
  -\frac{\bigl(\int_{\Ccap}fh \det(A[h])\dd\bigr)^2}
         {\int_{\Ccap}h \det(A[h])\dd}\right\}.
  \label{eq:local-af}
\end{equation}
Equality holds if and only if
\begin{equation}
  f=a+\sum_{\alpha=1}^{n}b_\alpha\frac{\xi_\alpha}{h},
  \label{eq:poincare-kernel}
\end{equation}
for some constants $a,b_\alpha\in\mathbb{R}$.
\end{lemma}

\begin{proof} 
Let
\[
  \bar f_h\coloneqq
  \frac{\int_{\Ccap}fh\det(A[h])\dd}
       {\int_{\Ccap}h\det(A[h])\dd},
  \qquad
  f_0\coloneqq f-\bar f_h,
  \qquad
  q\coloneqq hf_0.
\]
Since $\nabla_\mu f_0=0$ on $\partial\Ccap$, we have
\[
  \nabla_\mu q
  =f_0\nabla_\mu h+h\nabla_\mu f_0
  =\cot\theta\,q
  \qquad\text{on }\partial\Ccap.
\]
In particular, when $n=1$,
\[
  q'(-\theta)=-\cot\theta\,q(-\theta),
  \qquad
  q'(\theta)=\cot\theta\,q(\theta).
\]
For sufficiently small $t$, the perturbation $h+tq$ remains a capillary convex support function. Applying the capillary Alexandrov--Fenchel inequality \eqref{eq:function-AF} to $h+tq$ and $h$, and using
\[
  V(q,h,\ldots,h)=0,
\]
gives
\[
  V(q,q,h,\ldots,h)\le 0.
\]
The mixed-discriminant formula, followed by integration by parts using \eqref{eq:block-boundary} and the condition $\nabla_{\mu}f_0=0$ on $\partial\Ccap$, yields
\begin{equation}
  n(n+1)V(q,q,h,\ldots,h)
  =
  n\int_{\Ccap}f_0^2h\det(A[h])\dd
  -\int_{\Ccap}h^2U^{ij}(f_0)_i(f_0)_j\dd.
  \label{eq:second-variation}
\end{equation}
When $n=1$, identity \eqref{eq:second-variation} can also be verified directly. Since $q=hf_0$, $A[h]=h''+h$, and $f_0'(\pm\theta)=0$, integration by parts gives
\[
  2V(q,q)
  =\int_{-\theta}^{\theta}f_0^2h(h''+h)\,ds
  -\int_{-\theta}^{\theta}h^2(f_0')^2\,ds,
\]
which is exactly \eqref{eq:second-variation} for $n=1$.
Since $V(q,q,h,\ldots,h)\le 0$, this proves \eqref{eq:local-af}.

We now characterize the equality case. Assume equality holds in \eqref{eq:local-af}. Then \eqref{eq:second-variation} yields
$V(q,q,h,\ldots,h)=0$. Together with $V(q,h,\ldots,h)=0$, the quadratic expansion of
$V(h+tq,h+tq,h,\ldots,h)$ shows that equality holds in the capillary Alexandrov--Fenchel
inequality for $(h+tq,h)$ for all sufficiently small $t$. Hence, by Theorem~\ref{thm-mwwx},
\[
  q=bh+\sum_{\alpha=1}^{n}b_\alpha\,\xi_\alpha.
\]
Since $V(q,h,\ldots,h)=0$ and $V(\xi_\alpha,h,\ldots,h)=0$ for each $\alpha$ (horizontal
translations), we obtain $b\,V(h,h,\ldots,h)=0$, hence $b=0$. Recalling $q=h(f-\bar f_h)$,
we conclude that
\[
  f=\bar f_h+\sum_{\alpha=1}^{n}b_\alpha\frac{\xi_\alpha}{h},
\]
which is \eqref{eq:poincare-kernel}.

Conversely, if $f=a+\sum\limits_{\alpha=1}^{n}b_\alpha\xi_\alpha/h$, then $f$ satisfies
$\nabla_\mu f=0$ on $\partial\Ccap$ and $\bar f_h=a$, so $q=h(f-\bar f_h)=\sum b_\alpha\xi_\alpha$.
Since $A[\xi_\alpha]=0$, we have $A[q]=0$ and thus $V(q,q,h,\ldots,h)=0$. Plugging into
\eqref{eq:second-variation} gives equality in \eqref{eq:local-af}.
\end{proof}

For the model case $h=\ell$, Lemma~\ref{lem:local-af} yields the sharp weighted Poincar\'e inequality.

\begin{corollary}\label{cor:weighted-poincare-neumann}
For every $f\in C^{2}(\Ccap)$ with $\nabla_{\mu}f=0$ on $\partial\Ccap$,
\begin{equation}\label{eq:model-poincare}
  \int_{\Ccap}\ell^{2}|\nabla f |^{2}\,d\sigma
  \ge n\int_{\Ccap}(f-\bar f)^{2}\ell\,d\sigma,
  \qquad
  \bar f\coloneqq \frac{\int_{\Ccap} f\ell\,d\sigma}{\int_{\Ccap} \ell\,d\sigma}.
\end{equation}
Equality holds if and only if
\[
  f=a+\sum_{\alpha=1}^{n}b_{\alpha}\frac{\xi_{\alpha}}{\ell}
\]
for some constants $a,b_{\alpha}\in\R$. Equivalently, the first nonzero Neumann eigenvalue of
\[
  \cL_{\ell}f\coloneqq -\frac{1}{\ell}\,\diver(\ell^{2}\nabla f)
\]
is $n$, with eigenspace $\operatorname{span}\{\xi_1/\ell,\ldots,\xi_n/\ell\}$.
\end{corollary}

\begin{proof}
If $h=\ell$, then $A[\ell]=\sigma$, $\det(A[\ell])=1$, and
$U^{ij}=\sigma^{ij}$, so \eqref{eq:local-af} reduces to
\eqref{eq:model-poincare}.
\end{proof}

\subsection{A  Poincar\'e  inequality for the Dirichlet boundary condition}

The previous subsection treated the Neumann case. We now establish the corresponding inequality in the Dirichlet case. The key input is that the boundary defining function $\rho$ satisfies
\begin{eqnarray}\label{eq:rho}
  \nabla^{2}\rho+\rho\sigma=-\cos\theta\,\sigma,
\end{eqnarray}
which, in the acute-angle regime, produces a nonnegative remainder term.

\begin{lemma}\label{lem:hardy}
If $f\in C^{2}(\Ccap)$ and $f=0$ on $\partial\Ccap$, then
\begin{align}
  &\int_{\Ccap} h^{2}U^{ij}f_i f_j\,\dd
  -n\int_{\Ccap} f^{2}h\det(A[h])\,\dd\notag\\
  &=\int_{\Ccap} \rho^{2}U^{ij}
    \nabla_i\left(\frac{hf}{\rho}\right)
    \nabla_j\left(\frac{hf}{\rho}\right)\,\dd
   +\cos\theta\int_{\Ccap}\frac{\tr U}{\rho}(hf)^{2}\,\dd.
  \label{eq:hardy}
\end{align}
In particular,
\begin{equation}\label{eq:Dirichlet}
  \int_{\Ccap} h^{2}U^{ij}f_i f_j\,\dd
  \ge n\int_{\Ccap} f^{2}h\det(A[h])\,\dd
  +\cos\theta\int_{\Ccap}\frac{\tr U}{\rho}(hf)^{2}\,\dd.
\end{equation}
Equality holds if and only if $f=a\,h^{-1}\rho$ for some constant $a\in\R$.
\end{lemma}

\begin{proof}
Let $g\coloneqq hf$. Define $v$ by $g=v\rho$, that is, $v=g/\rho$. Since $f=0$ on $\partial\Ccap$, it follows that $g=0$ on $\partial\Ccap$ as well. We compute the integral of $U^{ij} g_ig_j$ in two ways.

\smallskip
\noindent\emph{Step 1.}
Expanding $g_i=h_if+hf_i$ gives
\begin{align*}
  U^{ij}g_ig_j
  &=U^{ij}(h_if+hf_i)(h_jf+hf_j)\\
  &=h^2U^{ij}f_if_j+2fh\,U^{ij}h_if_j+f^2U^{ij}h_ih_j.
\end{align*}
Integrating 
$\nabla_i (U^{ij}f^2 h h_j)$
over $\Ccap$ and using that $g=0$ on $\partial\Ccap$ (hence the boundary term vanishes), together with $\nabla_iU^{ij}=0$ from \eqref{eq:U-div-free}, yields
\begin{align*}
  0&=\int_{\Ccap}U^{ij}\nabla_i(f^2hh_j)\,\dd\\
   &=\int_{\Ccap}\Bigl(2fh\,U^{ij}h_if_j+f^2U^{ij}h_ih_j+f^2h\,U^{ij}h_{ij}\Bigr)\dd.
\end{align*}
Substituting this into the expansion of $\int U^{ij}g_ig_j$ gives
\[
  \int_{\Ccap}U^{ij}g_ig_j\dd
  =\int_{\Ccap}h^2U^{ij}f_if_j\dd-\int_{\Ccap}f^2h\,U^{ij}h_{ij}\dd.
\]
Finally, since $A[h]_{ij}=\nabla_{ij}h+h\sigma_{ij}$, we have
\[
  U^{ij}h_{ij}
  =U^{ij}A[h]_{ij}-h\,\tr U
  =n\det(A[h])-h\,\tr U,
\]
so
\begin{equation}
  \int_{\Ccap} h^2U^{ij}f_i f_j \dd-n\int_{\Ccap} f^2h \det(A[h]) \dd
  =\int_{\Ccap} U^{ij}g_i g_j \dd -\int_{\Ccap}(\tr U)g^2\dd.
  \label{eq:hardy-step1}
\end{equation}

\smallskip

\noindent\emph{Step 2.}
Since $\rho$ is a boundary defining function, the quotient
\[ v= \frac{g}{\rho}=\frac{hf}{\rho}\]
extends regularly to $\partial\Ccap$, and in particular $v\in C^{1}(\Ccap)$. Writing $g_i=\rho_i v+\rho v_i$, we obtain as in \textit{Step~1}
\begin{align*}
  U^{ij}g_i g_j
  &= \rho^{2}U^{ij}v_i v_j+2\rho v\,U^{ij}\rho_i v_j+v^{2}U^{ij}\rho_i\rho_j.
\end{align*}
Integrating the cross term by parts, and using $\nabla_iU^{ij}=0$ and $g=0$ on $\partial\Ccap$, we get
\[
  \int_{\Ccap}2\rho v\,U^{ij}\rho_i v_j\,\dd
  =-\int_{\Ccap}v^{2}\,U^{ij}(\rho\rho_{ij}+\rho_i\rho_j)\,\dd.
\]
Therefore
\begin{align*}
  \int_{\Ccap}U^{ij}g_i g_j\,\dd
  &= \int_{\Ccap}\rho^{2}U^{ij}v_i v_j\,\dd
   - \int_{\Ccap}\rho v^{2}\,U^{ij}\rho_{ij}\,\dd.
\end{align*}
Using $g^{2}=\rho^{2}v^{2}$ and rearranging,
\[
  \int_{\Ccap}U^{ij}g_i g_j\,\dd-\int_{\Ccap}(\tr U)g^{2}\,\dd
  =\int_{\Ccap}\rho^{2}U^{ij}v_i v_j\,\dd
   -\int_{\Ccap}\frac{g^{2}}{\rho}\,U^{ij}(\rho_{ij}+\rho\sigma_{ij})\,\dd.
\]
Combining this with \eqref{eq:hardy-step1} and \eqref{eq:rho} yields \eqref{eq:hardy}. The inequality \eqref{eq:Dirichlet}, together with its equality characterization, then follows immediately. 
\end{proof}

For the model $h=\ell$, Lemma~\ref{lem:hardy} yields the sharp weighted
Poincar\'e inequality under the Dirichlet boundary condition.
\begin{corollary}
If $f\in C^{2}(\Ccap)$ and $f=0$ on $\partial\Ccap$, then 
\begin{equation*}
  \int_{\Ccap}\ell^{2}|\nabla f |^{2}\,\dd
  \ge n\int_{\Ccap}f^{2}\ell\,\dd
  + n\cos\theta\int_{\Ccap}\frac{1}{\rho}f^{2}\ell^{2}\,\dd.
\end{equation*}
Equality holds if and only if $f=a\,\ell^{-1}\rho$ for some constant $a\in\R$.
\end{corollary}

The Corollary implies that the Dirichlet eigenvalue is strictly larger than $n$, when $\theta <\pi/2$. Moreover, the last term on the right hand side plays a crucial role, which comes from the right hand side of \eqref{eq:rho}.

\section{Proofs of Theorems \ref{prop:normalized} and \ref{thm:classification}}\label{sec4}
We now combine the estimates for the Neumann and Dirichlet modes obtained
in the preceding sections. The local capillary Alexandrov--Fenchel
inequality controls the first $n$ components of the conformal frame, while
Lemma~\ref{lem:hardy} provides the corresponding estimate for the last
component. Summing these estimates and using the completeness
identities \eqref{eq:zero-completeness} and \eqref{eq:first-completeness},
we obtain a reverse weighted Poincar\'e inequality for the capillary
support function $u=h/\ell$. This is the key estimate in the rigidity
argument.

To keep track of the mean and the contributions of the Neumann modes, set
\[
  M\coloneqq \int_{\Ccap}\ell\,\dd,
  \qquad
  \bar u\coloneqq \frac{1}{M}\int_{\Ccap}u\ell\,\dd,
  \qquad
  m_{\alpha}\coloneqq \int_{\Ccap}F_{\alpha}\ell\,\dd,
  \qquad
  m\coloneqq (m_1,\ldots,m_n).
\]

For $1\leq\alpha\leq n$, \eqref{eq:F-boundary} shows that $F_\alpha$
satisfies the Neumann boundary condition. Recall that
\begin{eqnarray}\label{eq:soliton-density}
  h\det A[h]=\ell.
\end{eqnarray}
Applying Lemma~\ref{lem:local-af} to $F_\alpha$, we obtain
\[
  \int_{\Ccap} h^{2}U^{ij}(F_{\alpha})_i(F_{\alpha})_j\,d\sigma
  \ge n\left(\int_{\Ccap}F_{\alpha}^{2}\ell\,d\sigma-\frac{m_{\alpha}^{2}}{M}\right).
\]
Summing over $1\le \alpha\le n$, we obtain
\begin{eqnarray}\label{eq:alpha}
  \sum_{\alpha=1}^{n}\int_{\Ccap} h^{2}U^{ij}(F_{\alpha})_i(F_{\alpha})_j\,d\sigma
  \ge n\left(\int_{\Ccap}\sum_{\alpha=1}^{n}F_{\alpha}^{2}\ell\,d\sigma-\frac{|m|^{2}}{M}\right).
\end{eqnarray}

Since \(F_{n+1}=0\) on \(\partial \Ccap\) and \(\rho\) is a smooth boundary defining function, the quotient \(F_{n+1}/\rho\) extends continuously to \(\partial \Ccap\). Equivalently, since \(F_{n+1}=-\tan\theta\,\zeta\), the quotient \(\zeta/\rho\) extends continuously as well. Thus all the integrands below involving \(1/\rho\) extend continuously to the boundary. For $F_{n+1}$, in view of \eqref{eq:F-boundary}, applying Lemma~\ref{lem:hardy} yields
\begin{eqnarray}
    \int_{\Ccap} h^{2}U^{ij}(F_{n+1})_i(F_{n+1})_j\,d\sigma
  \ge n\int_{\Ccap}F_{n+1}^{2}\ell\,d\sigma
  +\cos\theta\int_{\Ccap}\frac{\tr U}{\rho}\,h^{2}F_{n+1}^{2}\,d\sigma, 
\end{eqnarray}
which, together with \eqref{eq:alpha}, implies
\begin{equation}\label{eq:alpha1}
  \sum_{A=1}^{n+1}\int_{\Ccap} h^{2}U^{ij}(F_{A})_i(F_{A})_j\,d\sigma
  \ge n\left(\int_{\Ccap}\sum_{A=1}^{n+1}F_{A}^{2}\ell\,d\sigma-\frac{|m|^{2}}{M}\right) 
  +\cos\theta\int_{\Ccap}\frac{\tr U}{\rho}\,h^{2}F_{n+1}^{2}\,d\sigma.
\end{equation}

We next estimate the last term on the right-hand side of \eqref{eq:alpha1}, using part of the contribution from the left-hand side. To proceed, we write $A\coloneqq A[h]$, \eqref{eq:soliton-density} gives
\begin{eqnarray}\label{eq:cofactor-soliton}
  U=\det(A)A^{-1}
  =\frac{\ell}{h}A^{-1}.
\end{eqnarray}
It follows that
\[
h^2\tr U
=
h\ell\,\tr(A^{-1}).
\]
Together with \eqref{eq:F-horizontal}, we obtain
\begin{equation}\label{eq:nplusone}
      \cos\theta\int_{\Ccap}\frac{\tr U}{\rho}\,h^{2}F_{n+1}^{2}\,d\sigma
  =\int_{\Ccap}\frac{\sin^{2}\theta}{\cos\theta\,\rho}\,h\ell\,\zeta^{2}\,\tr(A^{-1})\,d\sigma.
\end{equation}
Combining \eqref{eq:alpha1}, \eqref{eq:nplusone}, and \eqref{eq:first-completeness}, we derive 
\begin{equation}
\int_{\Ccap}h^{2}U^{ij}B_{ik}B_j{}^{k}\,\dd\geq  n\left(\int_{\Ccap} \sum_{A=1}^{n+1}F_A^{2} \ell\,\dd-\frac{|m|^{2}}{M}\right)
  +\int_{\Ccap}\frac{\sin^{2}\theta}{\cos\theta\,\rho}\,h\ell\,\zeta^{2}\,\tr(A^{-1})\,\dd .
  \label{eq:combined}
\end{equation}
By \eqref{eq:zeta-B-def} and \eqref{eq:cofactor-soliton}, we have
\begin{equation}\label{eq:right}
  \int_{\Ccap} h^{2}U^{ij}B_{ik}B_j{}^{k}\,d\sigma
  =\int_{\Ccap}h\ell\,\bigl(\tr A-2n\zeta+\zeta^{2}\tr(A^{-1})\bigr)\,\dd.
\end{equation}
Recall that $\sin^{2}\theta=\ell+\cos\theta\,\rho$. Combining \eqref{eq:zero-completeness}, \eqref{eq:combined}, and \eqref{eq:right}, we obtain
\begin{equation}\label{eq:combined2}
\int_{\Ccap}h\ell\left(\tr A-2n\zeta-\frac{\ell \tr(A^{-1})}{\cos\theta\,\rho} \zeta^{2}\right) \dd\geq   n\left(\int_{\Ccap}(\ell^{2}|\nabla u|^{2}+\sin^{2}\theta\,u^{2})\,\ell\,\dd-\frac{|m|^{2}}{M}\right).
\end{equation}
Using $\tr(A^{-1})\tr A\geq n^2$, we get
\begin{eqnarray}
   -2n\zeta-\frac{\ell \tr (A^{-1})}{\cos\theta \rho}\zeta^2 &= & \frac{n^2\cos\theta \rho}{\ell \tr(A^{-1})}-\frac{\ell\tr(A^{-1})}{\cos\theta \rho} \left(\zeta+\frac{n\cos\theta \rho}{\ell \tr(A^{-1})} \right)^2 \notag
\\&  \le  &\frac{n^2\cos\theta \rho}{\ell \tr(A^{-1})}
  \le\frac{\cos\theta \rho}{\ell}\tr A.
    \label{eq:square-completion}
\end{eqnarray}
We remark that here we used $\cos\theta>0$ and $\rho>0$.
Substituting \eqref{eq:square-completion} into \eqref{eq:combined2}, we obtain
 
\begin{equation}
\sin^2\theta \int_{\Ccap}h\,\tr A\dd\geq   n\left(\int_{\Ccap}(\ell^2|\nabla u|^2+\sin^2\theta  u^2)\ell\dd
   -\frac{|m|^2}{M}\right).
  \label{eq:main-estimate}
\end{equation}
By \eqref{eq:trace-identity}, we have
\[
  \int_{\Ccap} h\,\tr A\,\dd
  = n\int_{\Ccap} u^{2}\ell\,\dd
  - \int_{\Ccap}\ell^{2}|\nabla u|^{2}\,\dd.
\]
Together with \eqref{eq:main-estimate}, this yields
\begin{equation}
  n\int_{\Ccap}\ell^3|\nabla u|^2\,\dd
  +\sin^2\theta\int_{\Ccap}\ell^2|\nabla u|^2\,\dd
  \le n\frac{|m|^2}{M}.
  \label{eq:energy-moment}
\end{equation}

We now estimate $m$. By the symmetry of $\Ccap$ in the horizontal variables and \eqref{eq:omega-def}, it is easy to see
\[
  \int_{\Ccap}\sin\theta\,\omega_{\alpha}\,\ell\,\dd
  =\sin^{2}\theta\int_{\Ccap}\xi_{\alpha}\,\dd=0,
  \qquad 1\leq\alpha\leq n.
\]
Hence we have
\[
  m_{\alpha}
  =\int_{\Ccap}\bigl(F_{\alpha}-\bar u\sin\theta\,\omega_{\alpha}\bigr)\ell\,\dd.
\]
By the Cauchy--Schwarz inequality and \eqref{eq:T-zero-completeness} with $c=\bar u$, we obtain
\begin{eqnarray}
  \frac{|m|^{2}}{M}
  &\le& \int_{\Ccap}\sum_{\alpha=1}^{n}
  \bigl(F_{\alpha}-\bar u\sin\theta\,\omega_{\alpha}\bigr)^{2}\ell\,\dd 
  \le \int_{\Ccap}\sum_{A=1}^{n+1}
  \bigl(F_A-\bar u\sin\theta\,\omega_A\bigr)^{2}\ell\,\dd \notag\\
  &=& \int_{\Ccap}\ell^{3}|\nabla u|^{2}\,\dd
        +\sin^{2}\theta\int_{\Ccap}(u-\bar u)^{2}\ell\,\dd.
  \label{eq:moment-estimate}
\end{eqnarray}
Combining \eqref{eq:energy-moment} and \eqref{eq:moment-estimate},  we obtain the desired reverse weighted Poincar\'e
inequality 
\begin{equation}
  \int_{\Ccap}\ell^2|\nabla u|^2\dd
  \le n\int_{\Ccap}(u-\bar u)^2\ell\dd.
  \label{eq:reverse-poincare}
\end{equation}
This is the decisive step in the rigidity argument. Indeed,
\eqref{eq:reverse-poincare} is the reverse of the sharp model inequality \eqref{eq:model-poincare}. Hence both inequalities must be
equalities, and the characterization of the equality case yields the
rigidity of the normalized support function.

\begin{proof}[\textbf{Proof of Theorem~\ref{prop:normalized}}]
By the preceding discussion, equality holds in \eqref{eq:model-poincare}.
Therefore, the equality characterization in Corollary \ref{cor:weighted-poincare-neumann} yields 
\[
  u=a+\sum_{\alpha=1}^{n}b_\alpha\frac{\xi_\alpha}{\ell},
  \qquad
  h=a\ell+\sum_{\alpha=1}^{n}b_\alpha\xi_\alpha.
\]
Using \eqref{eq:model-hessians}, we find
\[
  A[h]=a\sigma.
\]
Since $A[h]>0$, it follows that $a>0$. The equation
$h\det A[h]=\ell$ then reduces to
\[
  a^n\left(
    a\ell+\sum_{\alpha=1}^{n}b_\alpha\xi_\alpha
  \right)=\ell.
\]
Applying the linear operator
\[
  A[\cdot]=\nabla^2(\cdot)+(\cdot)\sigma
\]
to this identity gives $a^{n+1}=1$. Hence $a=1$, and the same identity
implies $b_\alpha=0$ for all $1\le\alpha\le n$. Therefore $h=\ell$.
\end{proof}

It remains only to translate the normalized rigidity statement back to
the original geometric formulation.

\begin{proof}[\textbf{Proof of Theorem~\ref{thm:classification}}]
Let
\begin{equation*}
  h_p(\xi)\coloneqq \langle X(\xi)-p,\nu(\xi)\rangle,
  \qquad \xi\in\Ccap.
\end{equation*}
If $p\in\partial\R^{n+1}_{+}$, then $\langle p,e\rangle=0$, and hence
\[
  h_p=h-\langle p,\xi\rangle.
\]
Since $A[\langle p,\xi\rangle]=0$, it follows that $A[h_p]=A[h]$. Moreover, \eqref{qeq:robin-bry} shows that $\langle p,\xi\rangle$ satisfies the same Robin boundary condition, so \eqref{eq:geometric-soliton} is equivalent to
\begin{equation*}
\left\{
\begin{array}{rlll}
  h_p\det A[h_p] &=& \lambda^{-1}\ell & \text{in }\Ccap,\\
  \nabla_{\mu}h_p &=& \cot\theta\,h_p & \text{on }\partial\Ccap.
\end{array}
\right.
\end{equation*}
In particular, $h_p>0$. Thus $\widetilde h\coloneqq \lambda^{1/(n+1)}h_p$ satisfies \eqref{eq:normalized-equation}, and Theorem~\ref{prop:normalized} gives
\[
  h_p=\lambda^{-1/(n+1)}\ell.
\]
Using the identity
\[
  \nabla\ell+\ell(\xi)(\xi-\cos\theta\,e)=\xi,
\]
we obtain
\[
  X-p=\lambda^{-1/(n+1)}\xi.
\]
This is precisely \eqref{eq:classification}, and the proof is complete.
\end{proof}

We conclude by indicating precisely where the acute-angle assumption enters the proof.

\begin{remark}\label{rem:acute-angle-necessary}
The acute-angle assumption is used essentially in the square-completion step \eqref{eq:square-completion}. Indeed, since
\[
  \nabla^{2}\rho+\rho\sigma=-\cos\theta\sigma,
\]
the last term on the right-hand side of \eqref{eq:hardy} is nonnegative for $\theta\in(0,\pi/2)$. When $\theta\in(\pi/2,\pi)$, this term changes sign, and the present argument does not appear to extend directly to the obtuse-angle regime.
\end{remark}


\begin{remark}
    Theorem \ref{thm:classification} can be viewed as a rigidity result for the logarithmic Minkowski problem  with prescribed anisotropic weight $\ell$.
\end{remark}

\bigskip

\appendix
\section{Proof of Theorem \ref{thm-mwwx} for the  planar case}\label{appendix}  

In this appendix, we provide, for completeness, a proof of Theorem \ref{thm-mwwx} in the planar case.

\medskip

\noindent{\it Proof of Theorem \ref{thm-mwwx} for $n=1$}. Let $f_{1}, f_{2}$ be positive, strictly convex capillary functions, and let $ K_{f_1}, K_{f_2}\subset \overline{\R^2_+} $ be the capillary convex bodies  whose support functions are $f_{1}$ and $f_{2}$. In view of \eqref{eqn:n=1-mixed-n-volume} and \cite[Section~5.1]{Schneider2014},  the standard mixed-area $V_2(K_{f_1},K_{f_2})$ of $K_{f_1},K_{f_2}$ satisfies

\[
  V_2(K_{f_1},K_{f_2})\coloneqq \frac{1}{2}\int_{-\theta}^{\theta}f_{1}(f^{''}_{2}+f_{2})ds=V(f_1,f_2).
\]
The classical planar Minkowski inequality
\cite[Theorem~7.2.1, Eq.~(7.18)]{Schneider2014} states that
\begin{eqnarray}\label{eqn:planar-minkowski}
      V(f_1,f_2)^2\geq \text{Vol}(K_{f_1})\text{Vol}(K_{f_2}),
\end{eqnarray}
with equality if and only if $K_{f_1}$ and $K_{f_2}$ are homothetic. Namely,
\begin{eqnarray}\label{eqn:n=1-equality-character}
    K_{f_1}=\lambda K_{f_2}+ z,
\end{eqnarray}for some $\lambda>0$ and $z\in \R^2$.

Now set
$$f_t\coloneqq f_1+tf.$$ Since $f_1>0$ and
$f_1''+f_1>0$, the function $f_t$ is positive, strictly convex and a capillary function for all sufficiently small $|t|$.   From \eqref{eqn:planar-minkowski}, we obtain
\[
  V(f_t, f_1)^2
  \geq \text{Vol}(K_{f_t})\text{Vol}(K_{f_1}),
\]
where $K_{f_t}$ is the capillary convex body with respect to the capillary function $f_{t}$.

Since $\text{Vol}(K_{f_t})=V_2(K_{f_t}, K_{f_t})=V_2(f_{t},f_{t})$, the mixed-area identification above
yields
\[
  V(f_t,f_1)^2\geq V(f_t,f_t)V(f_1,f_1).
\]
Using bilinearity, we obtain
\[
  \bigl(V(f_1,f_1)+tV(f,f_1)\bigr)^2
  \geq V(f_1,f_1)\left(V(f_1,f_1)+2tV(f,f_1)+t^2V(f,f)\right).
\]
After cancellation and division by $t^2$, this is precisely \eqref{eq:function-AF} for $n=1$.

If equality holds in \eqref{eq:function-AF}, the expansion above shows that
equality holds in the classical Minkowski inequality for $K_{f_t}$ and
$K_{f_1}$ for every sufficiently small $t\neq0$. By the equality
characterization in \eqref{eqn:n=1-equality-character}, the two convex bodies are homothetic. Thus there exist
$\lambda_t>0$ and $z_t\in\R^2$ such that
\[
  K_{f_t}=\lambda_tK_{f_1}+z_t.
\]
Taking the minimum of the
$E_2$-coordinate in this identity and using
\[
  \min_{x\in K_{f_t}}\langle x,E_2\rangle
  =\min_{x\in K_{f_1}}\langle x,E_2\rangle=0,
\]
we obtain $\langle z_t,E_2\rangle=0$, so $z_t=b_tE_1$ for some $b_t\in \R$. Taking the usual
support functions and restricting to the curved normals  gives
\[
  f_1+tf=\lambda_tf_1+b_t\xi_1,~~~~\xi\in \Ccap.
\]
Fixing one sufficiently small $t\neq0$ yields $$f=af_1+b\xi_1,$$ where $a\coloneqq \frac{\lambda_t-1}{t} \in \R,b\coloneqq \frac{b_t}{t}\in\R$. This is exactly \eqref{eqn:af-ineq-equality} for $n=1$.

\bigskip

Conversely, $A[\xi_1]=\xi_1''+\xi_1=0$, and symmetry of $V$ gives $V(\xi_1,g)=0$ for every capillary function $g$. Thus $f=af_1+b\xi_1$ implies
\[
  V(f,f_1)=aV(f_1,f_1),
  \qquad
  V(f,f)=a^2V(f_1,f_1),
\]
which proves equality and completes the proof. 
\qed

\bigskip

\noindent\textit{Acknowledgments.} G.W. would like to thank Chao Xia for many discussions on classification of solutions to the anisotropic log-Minkowski problem.
X.M. was supported by the Postdoctoral Fellowship Program of CPSF under Grant Numbers 2025T180843 and 2025M773082. L.W. was partially supported by CRM De Giorgi of Scuola Normale Superiore. L.W. is a member of GNAMPA as part of INdAM.

\end{document}